\documentclass[11pt]{article}
\usepackage[utf8]{inputenc}
\usepackage[T1]{fontenc}
\usepackage{latexsym,amssymb,amsmath,amsfonts,amsthm}
\usepackage{graphics}
\usepackage[section]{placeins}
\usepackage{graphicx}
\usepackage{mathrsfs}
\usepackage{subfigure}
\usepackage{float}
\usepackage{color}
\usepackage{mathtools}
\usepackage{tikz}
\usetikzlibrary{calc,intersections}
\usepackage{geometry}        
\usepackage{caption}         
\usepackage{enumitem}
\setlist[enumerate]{label=(\arabic*), font=\textup}
\usepackage[ruled]{algorithm2e}
\usepackage{float}
\usepackage{cleveref}
\usepackage{cases}
\usepackage{changepage}
\usepackage{zhlipsum}
\usepackage{epstopdf}
\newcommand{\R}{{\mathbb R}}

\newcommand{\be}{\begin{eqnarray}}
\newcommand{\ben}{\begin{eqnarray*}}
\newcommand{\en}{\end{eqnarray}}
\newcommand{\enn}{\end{eqnarray*}}

\newtheorem{theorem}{Theorem}[section]

\newtheorem{lemma}[theorem]{Lemma}

\definecolor{mgq}{rgb}{1,0,0}
\definecolor{ghx}{rgb}{0,0,0}
\definecolor{rot1}{rgb}{0,0,0}
\definecolor{rot}{rgb}{0,0,0}
\definecolor{rot2}{rgb}{0,0,0}
\definecolor{mgq1}{rgb}{0,0,0}
\definecolor{mgq2}{rgb}{0,0,0}

\begin{document}
	\renewcommand{\theequation}{\arabic{section}.\arabic{equation}}
	
	\title{A novel sampling method for reconstruction of a moving point acoustic source in $\mathbb{R}^3$}
	
	\author{Guanqiu Ma\footnotemark[3], Shuang Li\footnotemark[2], and Hongxia Guo\footnotemark[1] }

	\date{}
	\maketitle
	
	\renewcommand{\thefootnote}{\fnsymbol{footnote}}
	\footnotetext[3] {School of Mathematical Sciences, Sichuan Normal University, Chengdu, 610068, China (gqma@sicnu.edu.cn).}
	\footnotetext[2] {School of Mathematical Sciences, Sichuan Normal University, Chengdu, 610068, China (double@stu.sicnu.edu.cn).}
	\footnotetext[1]{Corresponding author: School of Mathematical Sciences, and Institute of Mathematics and Interdisciplinary Sciences, Tianjin Normal University, Tianjin, 300387, China (hxguo@tjnu.edu.cn).}

	\renewcommand{\thefootnote}{\arabic{footnote}}

	\begin{abstract}
		In this paper, we introduce a novel sampling method to recovering the trajectory of a moving point source in $\R^3$, where both the spatial location and emission moment of the moving point source are unknown. Combining algebraic theory with geometric knowledge, we prove the uniqueness of the source location by using measured data from five observation points. Our sampling method constructs an indicator function based on the property that the residual of the time difference of arrival constraint formula vanishes at the true source location. It achieves the reconstruction of the spatial positions and emission moments of a moving point source only using data from five observation points and their corresponding arrival times. The algorithm not only reduces the required number of observation points, but also improves computational efficiency, stability, and noise resistance. Numerical experiments verify the effectiveness of the method.
		
		\vspace{.2in} {\bf Keywords}: {\bf inverse moving source problem; uniqueness; sampling method; time difference of arrival}
	\end{abstract}

	\section{Introduction}
	
	In a homogeneous and isotropic medium occupying the whole space $\mathbb{R}^3$, we consider the dynamic wave propagation problem generated by a moving point source, where the wave speed $c > 0$ is constant. The source trajectory is denoted by $a(t) \in C(t_{\min}, t_{\max}) $, with $0 < t_{\min} < t_{\max}$, and {the source moves}
	at subsonic speed, satisfying
	\begin{equation}\label{speed-p}
		|a'(t)| < c, \quad  t \in [t_{\min}, t_{\max}].
	\end{equation}
	During the observation process, the moving point source emits a finite sequence of pulses described by
	\begin{equation}
		S(x,t) = \sum_{k=1}^{K} \delta(x-a(t_k)) \delta(t-t_k) \ell(t_k), \quad t_{\min} \leq t_1 < \cdots < t_K \leq t_{\max},
	\end{equation}
	where $\delta$ denotes the Dirac delta distribution. The real-valued continuous function $\ell(t)$ modulates the pulse intensity and satisfies $|\ell(t)| \geq \ell_0 >0$, which ensures that each emitted pulse is non-vanishing.
	
	The radiated acoustic field $U(x,t)$ satisfies the initial value problem 
	\begin{equation}
		\left\{
		\begin{aligned}
			&c^{-2}\frac{\partial^2 U}{\partial t^2} = \Delta U + S(x,t), \quad &&(x, t) \in \mathbb{R}^3 \times \mathbb{R}^+, \mathbb{R}^+ \coloneqq \{t\in \mathbb{R}: t>0\},\\
			&U(x,0)=\partial_t U(x,0) = 0, &&x\in \mathbb{R}^3.
		\end{aligned}
		\right.
	\end{equation}
	
	By convolving the source term $S(x,t)$ with the fundamental solution $G(x;t,c) := \frac{\delta(t-c^{-1}|x|)}{4\pi |x|}$, the explicit solution is given by
	\begin{align*}
		U(x,t)& = G(x;t,c) * S(x,t) \coloneqq \int_{\mathbb{R}^+} \int_{\mathbb{R}^3}
		G(x-y; t-\tau,c)S(y,\tau)\,dyd\tau\\
		&=\sum\limits_{k=1}^{K} \frac{\delta (t-t_k-c^{-1}|x-a(t_k)|)}{4\pi |x-a(t_k)|} \ell(t_k), \quad x \notin \Gamma,
	\end{align*}
	where $\Gamma \coloneqq \{x: x=a(t),\, t\in [t_{\min}, t_{\max}]\}$. According to Huygen's principle, for a given receiver location $\tilde{x}_i$, the signal is observed only at discrete \textbf{Times of Arrival}:
	\begin{equation}
		T_{ki} = t_k + c^{-1}|\tilde{x}_i-a(t_k)|,\quad k=1,\cdots,K \text.
	\end{equation}
	
	In this paper we are interested in the following inverse problem:
	\begin{description}
		\item[(IP):] Recovery the trajectory $\Gamma$ and emission moments $t_k,\, k=1,\cdots,K,$ using the time of arrival
		\[\{T_{ki}: k=1,2,\cdots,K; \, i=1,\cdots,I\}\]
		at the observation points $\{\tilde{x}_i: i=1,\cdots,I\}$, with $\tilde{x}_i \notin \Gamma$.
	\end{description}
	
	The problem of reconstructing the location and emission moments of a moving point source arises in a wide range of practical applications, including radar localization, navigation, communication transmission, sonar, and seismic monitoring. Numerical methods for inverting moving point sources based on arrival-time measurements can generally be categorized into two approaches: the time of arrival (TOA) method and the time difference of arrival (TDOA) method. The TOA method \cite{FKK1999,S1999} relies on absolute propagation times and requires prior knowledge of the emission moments. Theoretically, using the emission moment and the arrival time data from four non-coplanar observation points, one can construct to recover the position of the source at the emission moment. However, this method imposes stringent requirements on time synchronization between the source and observation points, as calibration errors can degrade inversion accuracy. In contrast, the TDOA method \cite{GM2015,TOY2021} eliminates the need for time synchronization between the source and observation points. It only requires that there be time synchronization among each observation point and offers greater stability in practical applications. Moreover, as long as any bias in the arrival times is consistent across all observation points, the localization result remains unaffected.
	
	Existing literature \cite{CNAS2014,GG2003,LQS2012,WSZG2019,DZY2020,TWC2022} has established a mature theoretical foundation for the TDOA method. These studies indicate that a consensus has been reached regarding TDOA localization in three-dimensional space: four observation points generally lead to non-unique solutions, whereas five points ensure a unique solution \cite{NK2025}. However, existing proofs have overlooked the singular case where the point source is situated on the perpendicular bisector plane of the baseline connecting two observation points. For numerical algorithms, existing methods commonly use least squares to formulate iterative schemes or solve matrix equations for point source localization. The Taylor series expansion method \cite{WHF1976} provides high-precision solutions through iterative refinement, Chan’s algorithm \cite{CH1905} offers a non-iterative closed-form solution based on the Least Squares (LS) criterion, and the Lagrange multiplier method \cite{Su2020} seeks to reduce the sensitivity of iterative approaches \cite{TKI2022} to initial guesses. In recent years, time-domain sampling methods \cite{Li2026} have demonstrated the capability to recover the spatial support of various sources using initial arrival times. However, research on sampling algorithms specifically utilizing TDOA measurements is limited. 
	To address these issues, this paper provides, from algebraic and geometric perspectives, a rigorous proof of the non-uniqueness with four observation points and the uniqueness with five points in three-dimensional space. Meanwhile, the indicator function is constructed to reconstruct the position and emission moment of the moving point source using the sampling method for five observation points.
	
	The remainder of the paper is organized as follows. In Section \ref{uniq}, the uniqueness for the proposed inverse problem is proved. In section \ref{alg}, the sampling algorithm for inverting the moving point source is introduced. Numerical experiments validating our method are reported in Section~\ref{num}. Finally, the paper is concluded and possible directions for future work are outlined in Section~\ref{con}.
	
	\section{Uniqueness}\label{uniq}
	Before proving uniqueness, we must first explain that the ordering of the arrival times is determined under the assumption \eqref{speed-p}.
	\begin{lemma}\label{time-s}
		Let $\tilde{x}_i$ be a fixed observation point and the source following a trajectory $a(t)$ satisfies the assumption \eqref{speed-p}. Then, the arrival times satisfy
		\begin{equation*}
			T_{ki} > T_{li} \quad \mbox{if and only if} \quad t_k>t_l.
		\end{equation*}
	\end{lemma}
	
	\begin{proof}
		Suppose that $t_k > t_l$ but $T_{ki} \leq T_{li}$. Consider the difference
		\[T_{li}-T_{ki} = t_l -t_k + c^{-1}(|\tilde{x}_i-a(t_l)|-|\tilde{x}_i-a(t_k)|).\] 
		By the triangular inequality,
		\[\begin{aligned}
			|\tilde{x}_i - a(t_l)| &= |\tilde{x}_i - a(t_k) + a(t_k) - a(t_l)|) \\
			& \leq |\tilde{x}_i - a(t_k)| + |a(t_k) - a(t_l)|.
		\end{aligned}\]
		Hence,
		\[T_{li}-T_{ki} \leq t_l -t_k + c^{-1}|a(t_k) - a(t_l)|.\]
		Since $a(t)$ is continuous, by the mean value theorem, there exists $t \in (t_l, t_k)$ such that
		\[|a(t_k) - a(t_l)| = |a'(t)|(t_k - t_l).\]
		By the subsonic speed $|a'(t)| < c$, we obtain 
		\[\begin{aligned}
			T_{li}-T_{ki} &\leq t_l -t_k + c^{-1}|a'(t)|(t_k - t_l) \\
			&= (t_k - t_l)(c^{-1}|a'(t)| - 1)<0, 
		\end{aligned}\] 
		which contradicts $T_{ki} \leq T_{li}$.
	\end{proof}
	
	Since the arrival times are ordered, the problem of multiple emission moments can be separated into problems of single emission moments. For a fixed emission moment $t_k$, we consider a single point source located at $s = (x, y, z)$.
	
	\begin{theorem}
		The position of the point source $s$ cannot be uniquely determined by the time differences of arrival among four distinct observation points $\tilde{x}_1, \tilde{x}_2, \tilde{x}_3, \tilde{x}_4$ in $\mathbb{R}^3$.
	\end{theorem}
	
	\begin{proof}
		Let $\tilde{x}_1$ be the origin of the coordinate system. Define the distance from the point source to the sensor as $r_i \coloneqq \lvert s - \tilde{x}_i \rvert $, and $R_{ij} \coloneqq r_i - r_j = |s-\tilde{x}_i| - |s-\tilde{x}_j|$ represents the distance difference from the point source to the i-th and j-th observation points. Note that $R_{ij}=|c \Delta t_{ij} |:= |c(T_{ki} - T_{kj})|$. Since $\tilde{x}_1 = (0,0,0)$, $r_1 = |s| = \sqrt{s^Ts}$.  
		For $R_{i1}$, we have $R_{i1} + r_1 = r_i, \, i=2,3, 4$. Substituting $r_i$ and squaring both sides yields 
		\[(R_{i1} + r_1)^2 = s^Ts - 2\tilde{x}_i^Ts + \tilde{x}_i^T\tilde{x}_i.\]
		It can be simplified to 
		\[2\tilde{x}_i^Ts + 2R_{i1}r_1 = \tilde{x}_i^T\tilde{x}_i - R_{i1}^2. \]
		This system can be rewritten in matrix form as
		\begin{equation}\label{matrix}
			Bu = b,
		\end{equation} 
		where
		\[B=
		\begin{bmatrix}
			2\tilde{x}_2^T & 2R_{21} \\
			2\tilde{x}_3^T & 2R_{31} \\
			2\tilde{x}_4^T & 2R_{41}
		\end{bmatrix}, 
		\quad 
		u = 
		\begin{bmatrix}
			s \\
			r_1
		\end{bmatrix}, 
		\quad \mbox{and }
		b = 
		\begin{bmatrix}
			\tilde{x}_2^T\tilde{x}_2 - R_{21}^2\\
			\tilde{x}_3^T\tilde{x}_3 - R_{31}^2\\
			\tilde{x}_4^T\tilde{x}_4 - R_{41}^2
		\end{bmatrix}.\]
		Since the vectors $\tilde{x}_2, \tilde{x}_3, \tilde{x}_4$ are not collinear, $\text{rank}(B) = 3$. Therefore, from formula~\eqref{matrix}, we obtain $u = u_0 + hv$, where $s = s_0 + hv_s$ and $r_1 = r_{1,0} + hv_r$ for $h \in \mathbb{R}$. Imposing the geometric constraint $r_1^2 = s^T s = x^2 + y^2 + z^2$ leads to
		\[(r_{1,0} + hv_r)^2 = (x_0 + hv_x)^2 + (y_0 + hv_y)^2 + (z_0 + hv_z)^2.\] 
		Simplification yields a quadratic equation in $h$:
		\[ah^2 + 2bh +c = 0,\]
		where $a = v_x^2 + v_y^2 + v_z^2 - v_r^2, \; b = x_0v_x + y_0v_y + z_0v_z - r_{1,0}v_r,$ and $c = x_0^2 + y_0^2 + z_0^2 - r_{1,0}^2$.
		Solving this quadratic formula gives two different roots $h_1$ and $h_2$. Substituting back into $u = u_0 + hv$ yields two source positions $s_1$ and $s_2$. Therefore, the point source position cannot be uniquely determined.
	\end{proof}

	To obtain a unique solution, we add another observation point and then discuss the uniqueness.
	
	\begin{lemma}
		Let $\tilde{x}_1, \tilde{x}_2 \in \mathbb{R}^3$ be two observation points, s be the source, and $c \geq 0$ be the wave speed. Denote $\Delta t_{12} = T_{k1} - T_{k2}$ as the time difference of arrival from the source to $\tilde{x}_1$ and $\tilde{x}_2$. 
		
		$(1)$ If $|\Delta t_{12}| = 0$, then the source $s$ lies on the perpendicular bisector plane of the line segment $\tilde{x}_1\tilde{x}_2$.
		
		$(2)$ If $|\Delta t_{12}| = \frac{|\tilde{x}_1 - \tilde{x}_2|}{c}$, then $s$ lies on the line through $\tilde{x}_1$ and $\tilde{x}_2$.
		
		$(3)$ If $0 < |\Delta t_{12}| < \frac{|\tilde{x}_1 - \tilde{x}_2|}{c}$, then the formula $|s - \tilde{x}_1| - |s - \tilde{x}_2| = c \Delta t_{12}$ uniquely determines one branch of a two-sheet hyperboloid with foci $\tilde{x}_1$, $\tilde{x}_2$.
	\end{lemma}
	
	\begin{proof}
		By the triangular inequality, for any $s \in \mathbb{R}^3$, we have $\big ||s - \tilde{x}_1| - |s - \tilde{x}_2| \big | < |\tilde{x}_1 - \tilde{x}_2|$. 
		
		$(1)$ If $|\Delta t_{12}| = 0$, then $|s - \tilde{x}_1| = |s - \tilde{x}_2| $, so $s$ lies on the perpendicular bisector plane of the line segment $\tilde{x}_1\tilde{x}_2$.
		
		$(2)$ If $|\Delta t_{12}| = \frac{|\tilde{x}_1 - \tilde{x}_2|}{c}$, then $|s - \tilde{x}_1| - |s - \tilde{x}_2| = |\tilde{x}_1 - \tilde{x}_2|$, which implies $s$ lies on the line through $\tilde{x}_1$ and $\tilde{x}_2$.
		
		$(3)$ If $0 < |\Delta t_{12}| < \frac{|\tilde{x}_1 - \tilde{x}_2|}{c}$, then the set of points satisfying the given constant distance difference is a two-sheet hyperboloid. The sign of $\Delta t_{12}$ selects a single branch.
	\end{proof}
	
	Next, we will prove the uniqueness of the point source with five observation points in a specific configuration in $\R^3$.
	
	\begin{theorem}
		Suppose there are five distinct observation points $\tilde{x}_1, \tilde{x}_2, \tilde{x}_3, \tilde{x}_4, \tilde{x}_5$ in $\mathbb{R}^3$. Let $\tilde{x}_1$ be the origin, $\tilde{x}_2 = (a,0,0)$, $\tilde{x}_3 = (-b,0,0)$ with $a,b\in \R^+$, and assume that $\tilde{x}_1, \tilde{x}_2, \tilde{x}_3$ are collinear on the $x$-axis. Let $\tilde{x}_4 = (0,d,0)$ on the $y$-axis and $\tilde{x}_5 = (0,0,f)$ on the $z$-axis, with $d,f \in \mathbb{R}^+$. Then the position of the point source $s$ can be uniquely determined by the time differences of arrival, from which the emission moment $t_k = T_{k1} - c^{-1}|s - \tilde{x}_1|$ can also be uniquely obtained.
	\end{theorem}
	
	\begin{proof}
		
		From physical background, the point source $s$ exists.
		
		\textbf{Step 1}: According to the time differences of arrival $\Delta t_{21} \mbox{ and } \Delta t_{31}$, the source $s$ satisfy
		\begin{numcases}{ }
			|s - \tilde{x}_2| - |s - \tilde{x}_1| = c\Delta t_{21},  \label{eq:t21} \\
			|s - \tilde{x}_3| - |s - \tilde{x}_1| = c\Delta t_{31}.  \label{eq:t31}
		\end{numcases}
		
		Case 1: $\Delta t_{21} = 0$ and $\Delta t_{31} = 0$. 
		
		Both equations define perpendicular bisector planes of $\tilde{x}_1\tilde{x}_2$ and $\tilde{x}_1\tilde{x}_3$. Since $\tilde{x}_1, \tilde{x}_2$, and $\tilde{x}_3$ are distinct and collinear on the $x$-axis, these planes are parallel and disjoint. This case is excluded by the physical background.
		
		Case 2: Either $\Delta t_{21} = 0$ with $0 < |\Delta t_{31}| < \frac{|\tilde{x}_1-\tilde{x}_3|}{c}$, or $0 < |\Delta t_{21}| < \frac{|\tilde{x}_1-\tilde{x}_2|}{c}$ with $\Delta t_{31} = 0$. 
		
		Without loss of generality, we focus on the first one. As shown in Figure~\ref{g1}, formula~\eqref{eq:t21} gives the perpendicular bisecting plane of $\tilde{x}_1 \tilde{x}_2$, and formula~\eqref{eq:t31} describes one sheet of the two-sheet hyperboloid with foci $\tilde{x}_1$, $\tilde{x}_3$. The three collinear distinct points on the x-axis imply that the intersection is a circle perpendicular to the $x$-axis: 
		\[y^2 + z^2 = m_1^2, \; x = \frac{a}{2} \neq 0,\]
		where 
		\[
		m_1^2 = \frac{\left(b^2 - (c \cdot \Delta t_{31})^2\right)\left((a+b)^2 - (c \cdot \Delta t_{31})^2\right)}{4(c \cdot \Delta t_{31})^2}.
		\]
		\begin{figure}[htbp]
			\centering
			\begin{minipage}[b]{0.45\textwidth}
				\centering
				\includegraphics[width=0.65\textwidth]{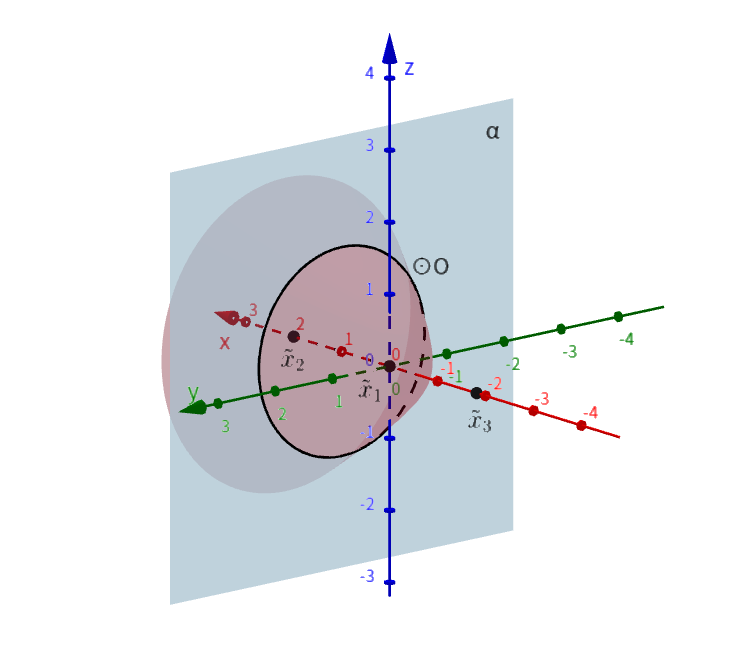}
				\caption{The hyperboloid and the perpendicular bisector plane intersect in $\odot O$.}
				\label{g1}
			\end{minipage}
			\hfill
			\begin{minipage}[b]{0.45\textwidth}
				\centering
				\noindent\includegraphics[width=0.88\textwidth]{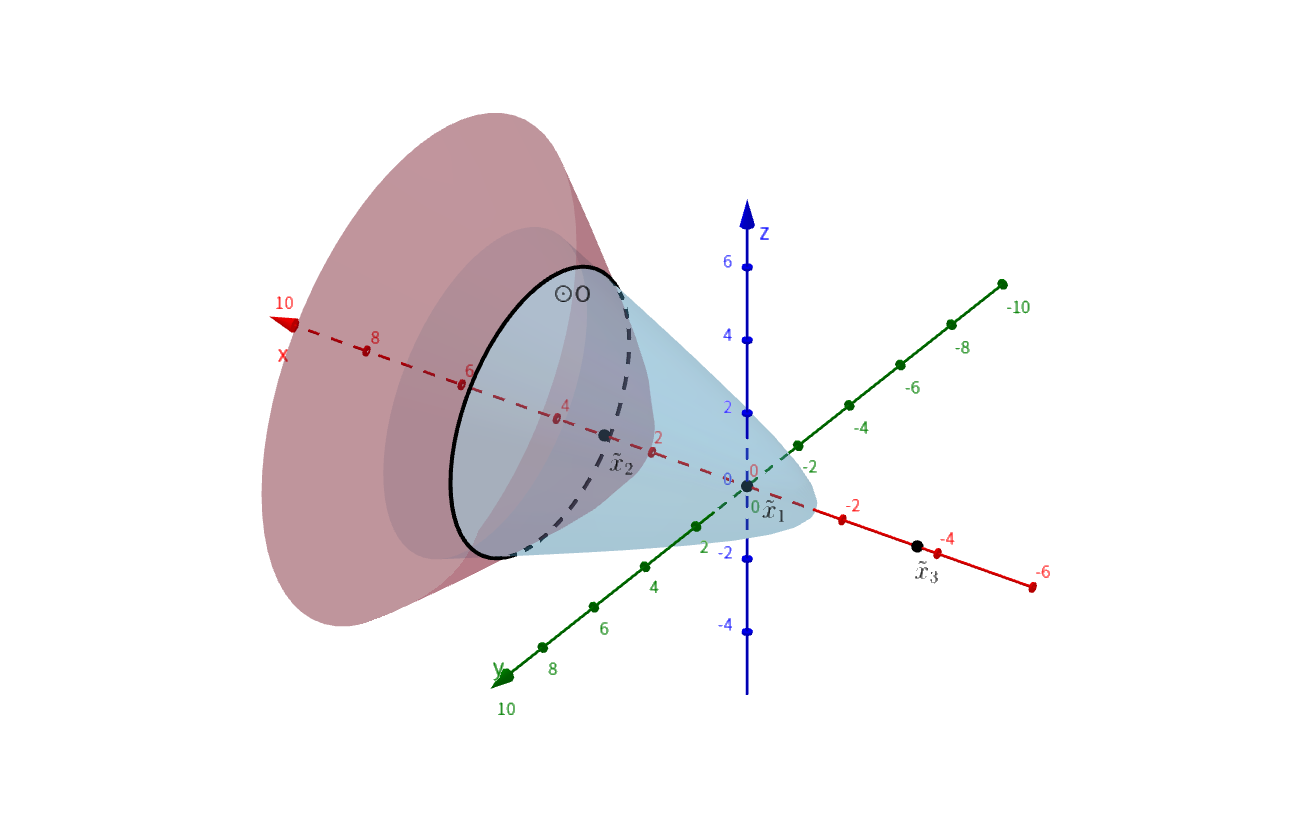}
				\caption{The two hyperboloids intersect in $\odot O$.}
				\label{g2}
			\end{minipage}
		\end{figure}
		
		Case 3: $\Delta t_{21} = 0$ with $|\Delta t_{31}| = \frac{|\tilde{x}_1-\tilde{x}_3|}{c}$,or $|\Delta t_{21}| = \frac{|\tilde{x}_1-\tilde{x}_2|}{c}$ with $\Delta t_{31} = 0$. 
		
		Without loss of generality, we only discuss the former. Formula~\eqref{eq:t21} defines the perpendicular bisecting plane of the line segment $\tilde{x}_1\tilde{x}_2$. 
		
		\begin{enumerate}[label=(\arabic*), leftmargin=*]
			\item If $\Delta t_{31} > 0$, formula~\eqref{eq:t31} describes a ray emanating from $\tilde{x}_1$ along the positive x-axis. This ray intersects the plane at a point on the positive x-axis, which gives a unique solution.
			\item If $\Delta t_{31} < 0$, formula~\eqref{eq:t31} describes a ray emanating from $\tilde{x}_3$ along the negative x-axis. The ray does not intersect the plane, so the assumption is invalid.
		\end{enumerate}
		
		Case 4: $0 < |\Delta t_{21}| < \frac{|\tilde{x}_1-\tilde{x}_2|}{c}$ and $0 < |\Delta t_{31}| < \frac{|\tilde{x}_1-\tilde{x}_3|}{c}$. 
		
		As shown in Figure~\ref{g2}, since $\tilde{x}_1, \tilde{x}_2$, and $\tilde{x}_3$ are collinear along the $x$-axis, each formula defines one branch of a two-sheet hyperboloid of revolution about the $x$-axis. For distinct $\tilde{x}_1, \tilde{x}_2, \tilde{x}_3$, if the two hyperboloids intersect, their intersection is at most a circle perpendicular to the $x$-axis: 
		\[y^2 + z^2 = m_2^2, \; x = x_2,\]
		where 
		\begin{equation}\label{eq:x2}
			x_2 = \frac{a^2(c\Delta t_{31}) - b^2(c\Delta t_{21}) + c^3\Delta t_{21}\Delta t_{31}(\Delta t_{31} - \Delta t_{21})}{2c(a\Delta t_{31} + b\Delta t_{21})},
		\end{equation}
		and 
		\[m_2^2 = \left[ \frac{a^2 - (c\Delta t_{21})^2}{2c\Delta t_{21}} - \frac{a}{c\Delta t_{21}}x_0 \right]^2 - x_0^2.\]
		If two hyperboloids intersect at only one point, then uniqueness is determined. If two hyperboloids do not intersect, it is known from the physical background that the assumption does not hold.
		
		Case 5: $0 < |\Delta t_{21}| < \frac{|\tilde{x}_1-\tilde{x}_2|}{c}$ with $|\Delta t_{31}| = \frac{|\tilde{x}_1-\tilde{x}_3|}{c}$, or $|\Delta t_{21}| = \frac{|\tilde{x}_1-\tilde{x}_2|}{c}$ with $0 < |\Delta t_{31}| < \frac{|\tilde{x}_1-\tilde{x}_3|}{c}$. 
		
		Without loss of generality, and we only discuss the former. Formula~\eqref{eq:t21} defines one sheet of the two-sheet hyperboloid with foci $\tilde{x}_1$, $\tilde{x}_2$, which intersects the $x$-axis on its positive half.
		
		\begin{enumerate}[label=(\arabic*), leftmargin=*]
			\item If $\Delta t_{31} > 0$, formula~\eqref{eq:t31} represents a ray emanating from $\tilde{x}_1$ along the positive $x$-axis. This ray intersects the hyperboloid at exactly one point.
			\item If $\Delta t_{31} < 0$, formula~\eqref{eq:t31} represents a ray emanating from $\tilde{x}_3$ along the negative $x$-axis. The ray has no intersection with the hyperboloid, so the assumption does not hold.
		\end{enumerate}
		
		Case 6: $|\Delta t_{21}| = \frac{|\tilde{x}_1-\tilde{x}_2|}{c}$ and $|\Delta t_{31}| = \frac{|\tilde{x}_1-\tilde{x}_3|}{c}$.
		
		\begin{enumerate}[label=(\arabic*), leftmargin=*]
			\item If $\Delta t_{21} > 0$ and $\Delta t_{31} > 0$, formula~\eqref{eq:t21} defines a ray emanating from $\tilde{x}_1$ along the negative $x$-axis, and formula~\eqref{eq:t31} defines a ray emanating from $\tilde{x}_1$ along the positive $x$-axis. Since $\tilde{x}_1 \notin \Gamma$, this case is excluded by the problem assumption.
			
			\item If $\Delta t_{21} > 0$ and $\Delta t_{31} < 0$, formula~\eqref{eq:t21} defines a ray emanating from $\tilde{x}_1$ along the negative $x$-axis, and formula~\eqref{eq:t31} defines a ray emanating from $\tilde{x}_3$ along the negative $x$-axis. Since $\tilde{x}_3 \notin \Gamma$, the intersection is an open ray starting at $\tilde{x}_3$ and extending along the negative $x$-axis.
			
			\item If $\Delta t_{21} < 0$ and $\Delta t_{31} > 0$, formula~\eqref{eq:t21} defines a ray emanating from $\tilde{x}_2$ along the positive $x$-axis, and formula~\eqref{eq:t31} defines a ray emanating from $\tilde{x}_1$ along the positive $x$-axis. Since $\tilde{x}_2 \notin \Gamma$, the intersection is an open ray starting at $\tilde{x}_2$ and extending along the positive $x$-axis.
			
			\item If $\Delta t_{21} < 0$ and $\Delta t_{31} < 0$, formula~\eqref{eq:t21} defines a ray emanating from $\tilde{x}_2$ along the positive $x$-axis, and formula~\eqref{eq:t31} defines a ray emanating from $\tilde{x}_3$ along the negative $x$-axis. The two rays have no intersection, so the assumption does not hold.
		\end{enumerate}
		
		From the above discussion, Cases $2$ and $4$ each give a circle perpendicular to the $x$-axis. In Case $6$, when $\Delta t_{21}$ and $\Delta t_{31}$ have opposite signs, a open ray is obtained. The remaining cases either have a unique solution or are rejected due to physical background and problem assumption. 
		
		\textbf{Step 2}: According to the time difference of arrival $\Delta t_{41}$, the point source $s$ should satisfy
		\begin{equation}\label{eq:t41}
			|s - \tilde{x}_4| - |s - \tilde{x}_1| = c\Delta t_{41}.
		\end{equation}
		We analyze the undetermined cases.

		Case 1'(open ray from Case 6): The point source lies on an open ray. Without loss of generality, we discuss (2). In this case, the point source should lie on the open ray starting from $\tilde{x}_3$ and pointing toward the negative $x$-axis.
		\begin{enumerate}[label=(\arabic*), leftmargin=*]
			\item If $\Delta t_{41} = 0$, formula~\eqref{eq:t41} defines the perpendicular bisecting plane of the line segment $\tilde{x}_1\tilde{x}_4$. This plane is parallel to the $xOz$ and has no intersection with the $x$-axis. Therefore, this case is excluded by the physical background.
			\item If $0 < |\Delta t_{41}| < \frac{|\tilde{x}_1-\tilde{x}_4|}{c}$, formula~\eqref{eq:t41} defines one sheet of the two-sheet hyperboloid with foci $\tilde{x}_1$, $\tilde{x}_4$, which is symmetric about the $y$-axis. Substituting $y = 0$ and $z = 0$ into formula~\eqref{eq:t41} yields
			\[\sqrt{(x - x_4)^2} - \sqrt{x^2} = c \Delta t_{41}.\]
			Solving this gives
			\[s_1 = (x_0,0,0), \; s_2 = (-x_0,0,0).\]
			Due to the constraints in (2) of Case 6, if $x_0 > b$, then a unique solution $s_1$ is determined. If $x_0 < b$, there is no intersection, and this case is excluded by the physical background.
			\item If $|\Delta t_{41}| = \frac{|\tilde{x}_1-\tilde{x}_4|}{d}$, formula~\eqref{eq:t41} either defines a ray starting from $\tilde{x}_1$ along the negative $y$-axis, or a ray starting from $\tilde{x}_4$ along the positive $y$-axis. These rays have no intersection with the open ray starting from $\tilde{x}_3$ along the negative $x$-axis. Hence, this case is excluded by the physical background.
		\end{enumerate}

		Case 2' (circle from Cases 2 or 4): The circle is given by $y^2 + z^2 = m^2, \; x = x_0$, where $m^2$ is explicitly defined as:
		\[m^2 = \begin{cases} 
				m_1^2, & \text{if derived from Case 2 (where } x_0 = \frac{a}{2}\text{),} \\ 
				m_2^2, & \text{if derived from Case 4 (where } x_0 \text{ is determined by Equation~\eqref{eq:x2}),}
			\end{cases}\]
		and $m_1^2, m_2^2$ are the explicit expressions obtained in Step 1.
		\begin{enumerate}[label=(\arabic*), leftmargin=*]
			\item If $\Delta t_{41} = 0$, as shown in Figure~\ref{g3}, formula~\eqref{eq:t41} defines the perpendicular bisecting plane of the line segment $\tilde{x}_1\tilde{x}_4$. This plane is parallel to the $xOz$ plane, and its expression is $y = \dfrac{d}{2}$.
			We have
			\[(\dfrac{d}{2})^2 + z^2 = m^2.\]
			Thus,
			\[z^2 = m^2 - \dfrac{d^2}{4}.\]
			When $m^2 > \dfrac{d^2}{4}$, we obtain $s_1 = (x_0,\dfrac{d}{2},\sqrt{m^2 - \dfrac{d^2}{4}}), \; s_2 = (x_0,\dfrac{d}{2},-\sqrt{m^2 - \dfrac{d^2}{4}})$ which are symmetric with respect to the $xOy$ plane. When $m^2 = \dfrac{d^2}{4}$, we obtain $s = (x_0,\dfrac{d}{2},0)$, and the solution is uniquely determined. When $m^2 < \dfrac{d^2}{4}$, there is no solution, and this case is excluded by the physical background.
	
			\begin{figure}[htbp]
				\centering
				\begin{minipage}[b]{0.45\textwidth}
					\centering
					\noindent\includegraphics[width=0.65\textwidth]{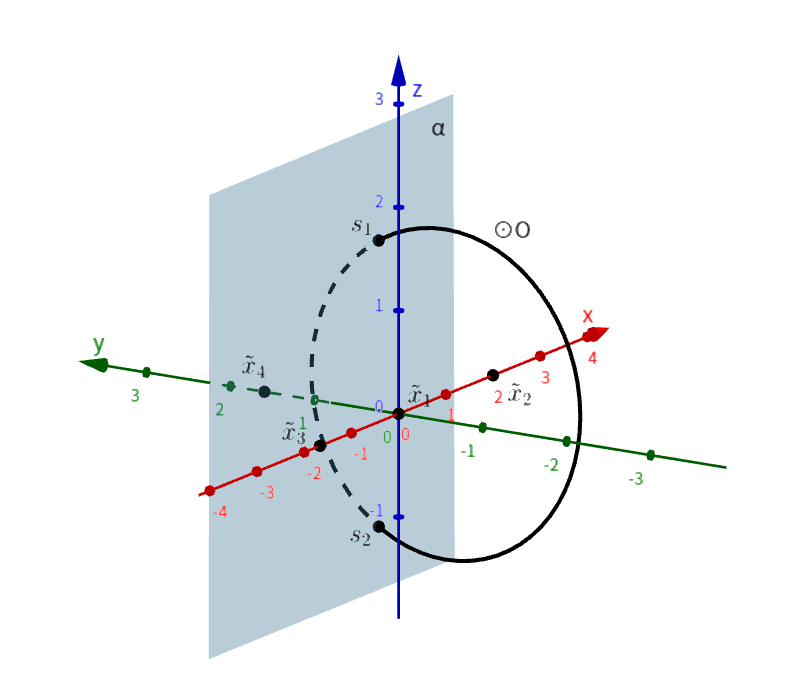}
					\caption{$\odot O $ intersect the perpendicular bisector plane at $s_1$ and $s_2$.}
					\label{g3}
				\end{minipage}
				\hfill
				\begin{minipage}[b]{0.45\textwidth}
					\centering
					\noindent\includegraphics[width=0.65\textwidth]{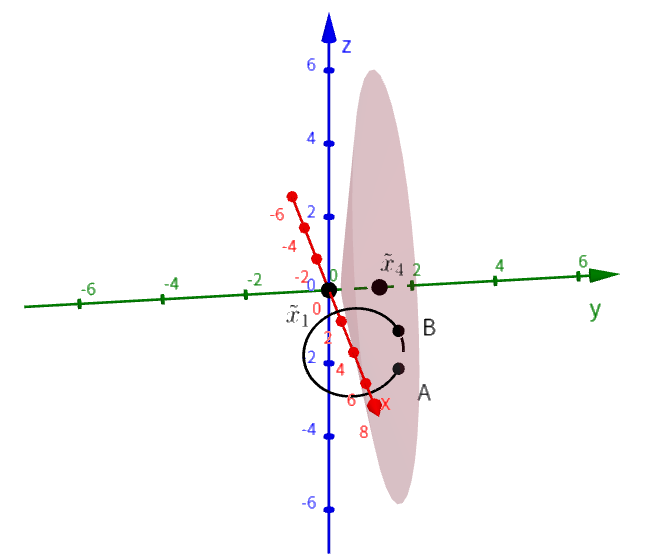}
					\caption{$\odot O $ intersect the hyperboloid at $s_1$ and $s_2$.}
					\label{g4}
				\end{minipage}
			\end{figure}
			
			\item If $0 < |\Delta t_{41}| < \frac{|\tilde{x}_1-\tilde{x}_4|}{c}$, as shown in Figure~\ref{g4}, formula~\eqref{eq:t41} defines one sheet of the two-sheet hyperboloid with foci $\tilde{x}_1$, $\tilde{x}_4$. System of formulas
			\[
			\begin{cases}
				|s - \tilde{x}_4| - |s - \tilde{x}_1| = c\Delta t_{41}, \\
				y^2 + z^2 = m^2, \; x = x_0.
			\end{cases}
			\]
			We have
			\[
			\begin{cases}
				\sqrt{(x_0 - x_4)^2 + (y - y_4)^2 + (z - z_4)^2} - \sqrt{x_0^2 + y^2 + z^2} = c\Delta t_{41},\\
				y^2 + z^2 = m^2.
			\end{cases}
			\]
			By algebraic elimination, $y$ is uniquely determined as a constant $y_0$. Substituting it into the circle's formula gives:
			\[z^2 = m^2 - y_0^2.\]
			When $m^2 > y_0^2$, we obtain $s_1 = \big(x_0, y_0, \sqrt{m^2 - y_0^2}\big)$ and $s_2 = \big(x_0, y_0, -\sqrt{m^2 - y_0^2}\big)$. The two solutions are symmetric about the $xOy$ plane. When $m^2 = y_0^2$, we have a unique solution $s = (x_0, y_0, 0)$. The inequality $m^2 < y_0^2$ leads to no solution, which is invalid given the physical background.
			
			\item If $|\Delta t_{41}| = \frac{|\tilde{x}_1-\tilde{x}_4|}{d}$, formula~\eqref{eq:t41} either defines a ray starting from $\tilde{x}_1$ along the negative $y$-axis, or a ray starting from $\tilde{x}_4$ along the positive $y$-axis. If the circle is the intersection of the two hyperboloids, then it is possible that $x = x_0 = 0$. We discuss this case separately:
			
			\begin{enumerate}[label=(\roman*)]
				\item When $\Delta t_{41} > 0$, formula~\eqref{eq:t41} defines a ray that starts at $\tilde{x}_1$ and points toward the negative $y$-axis. In this case, the circle intersects the ray at the point $\tilde{x}_1$, which is a contradiction to $\tilde{x}_1 \notin \Gamma$. Therefore, this case is excluded by the problem assumption.
				\item When $\Delta t_{41} < 0$, formula~\eqref{eq:t41} defines a ray that starts at $\tilde{x}_4$ and points toward the positive $y$-axis. In this case, the circle and the ray have no intersection, and this case is excluded by the physical background.
			\end{enumerate}
			
			If \(x = x_0 \neq 0\), the rays lie on the $y$-axis and have no intersection with the circle $y^2 + z^2 = m^2,\; x = x_0$. Therefore, this case is excluded by the physical background.
		\end{enumerate}	
		
		From the above discussion, we find that Case 2' determines two points symmetric with respect to the $xOy$. Case 1' either has a unique solution or is rejected due to physical background and problem assumption. Next, we will discuss the case where the outcome has not yet been determined.
		
		\textbf{Step 3}: According to the time difference of arrival $\Delta t_{51}$, the point source is required to satisfy 
		\begin{equation}\label{t_51}
			|s - \tilde{x}_5| - |s - \tilde{x}_1| = c\Delta t_{51}.
		\end{equation}
		If there is no unique solution, this implies that $s_1$ and $s_2$ must both satisfy formula~\eqref{t_51}. Substituting them gives
		\[
		\begin{cases}
			\sqrt{(x_0 - x_5)^2 + (y_0 - y_5)^2 + (z_0 - z_5)^2} = c\Delta t_{51} + \sqrt{x_0^2 + y_0^2 + z_0^2},\\
			\sqrt{(x_0 - x_5)^2 + (y_0 - y_5)^2 + (-z_0 - z_5)^2} = c\Delta t_{51} + \sqrt{x_0^2 + y_0^2 + z_0^2}.
		\end{cases}
		\]
		Solving the system we obtain
		\[z_0 = 0 \; \text{or} \; z_5 = 0.\] 
		When $z_0 = 0$, the point sources $s_1$ and $s_2$ coincide, which is a contradiction. When $z_5 = 0$, the observation points $\tilde{x}_1$ and $\tilde{x}_5$ coincide, which is also a contradiction. 
		
		In summary, the five observation points uniquely determine the position of the point source. Depending on the specific values of the time differences of arrival, the uniqueness can be systematically divided into the following four cases:
		\begin{enumerate}[label={(\roman*)}]
			\item {\bf Uniqueness by the first three observers ($\tilde{x}_1, \tilde{x}_2, \tilde{x}_3$)} \\
			Under Case 3(1) or Case 5(1), the geometric loci intersect at a single point on the $x$-axis, yielding a unique solution immediately.
			
			\item \textbf{Uniqueness by adding the fourth observer ($\tilde{x}_4$) on an open ray} \\
			Under Case 6(2) or Case 6(3), the first three observers yield an open ray, which is cut at a unique point by the two-sheet hyperboloid with foci $\tilde{x}_1$, $\tilde{x}_4$ (when $x_0 > b$). 
			
			\item \textbf{Uniqueness by adding the fourth observer ($\tilde{x}_4$) on a tangent circle} \\
			When the first three observers determine a spatial circle, if $m^2 = \frac{d^2}{4}$ or $m^2 = y_0^2$, the circle is tangent to the surface from $\tilde{x}_4$, fixing a unique solution on the $xOy$ plane.
			
			\item \textbf{Uniqueness by adding the fifth observer ($\tilde{x}_5$) to break symmetry} \\
			When the circle and the surface from $\tilde{x}_4$ intersect at two symmetric points ($m^2 > \frac{d^2}{4}$ or $m^2 > y_0^2$), the fifth observer $\tilde{x}_5$ on the $z$-axis breaks the symmetry to give the final unique solution.
		\end{enumerate}
		
		Thus, the emission moment $t_k$ is also uniquely determined by \[t_k = T_{k1} - c^{-1}|s - \tilde{x}_1|. \qedhere\]
	\end{proof}

	\section{Algorithm}\label{alg}
	
	In this section, we present a novel sampling method to calculate the locations $a(t_k)$ of the point source and the emission moments $t_k$, $k=1,\cdots,K$. 
	
	Let $\Omega \coloneqq [x_{\min},x_{\max}] \times [y_{\min},y_{\max}] \times [z_{\min},z_{\max}]$ be a bounded closed domain in $\mathbb{R}^3$. Partition $\Omega$ with a discrete grid step size $h > 0$:
	\[
	\begin{aligned}
		x_p &= x_{\min} + p \cdot h, \quad p = 0, 1, \dots, K_x, \\
		y_q &= y_{\min} + q \cdot h, \quad q = 0, 1, \dots, K_y, \\
		z_r &= z_{\min} + r \cdot h, \quad r = 0, 1, \dots, K_z,
	\end{aligned}
	\]
	where $K_x = \dfrac{x_{\max} - x_{\min}}{h}$, and similarly for $K_y, \; K_z$. The total number of grid points is $(K_x+1)(K_y+1)(K_z+1)$. Denote the grid points in $\Omega$ by $s_{pqr}$. 
	
	For any grid point $s_{pqr}\in \Omega$, using the relation 
	\begin{equation}
		|s - \tilde{x}_i| - |s - \tilde{x}_j| = c\Delta t_{ij},\label{eq:tij}
	\end{equation}
	we define a auxiliary indicator function 
	\[I_{ij}(s_{pqr}) \coloneqq \big| |s_{pqr}-\tilde{x}_i|-|s_{pqr}-\tilde{x}_j|-c\Delta t_{ij} \big|.\]
	The auxiliary indicator function points out the absolute difference between the theoretical distance difference and the measured distance difference at $s_{pqr}$. 
	
	\begin{lemma}
		For the formula~\eqref{eq:tij} and any spatial point $s_{pqr} \in \Omega$, we have
		\[
		I_{ij}(s_{pqr}) = \begin{cases}
			0, & \text{$s_{pqr}$ satisfies formula~\eqref{eq:tij}}, \\
			> 0, &\text{otherwise}.
		\end{cases}
		\]
	\end{lemma}
	
	\begin{proof}
		If $s_{pqr}$ satisfies formula~\eqref{eq:tij}, then $|s_{pqr} - \tilde{x}_i| - |s_{pqr} - \tilde{x}_j| = c\Delta t_{ij}$. Hence $I_{ij}(s_{pqr}) = \big| |s_{pqr}-\tilde{x}_i|-|s_{pqr}-\tilde{x}_j|-c\Delta t_{ij} \big| = 0$. If $s_{pqr}$ does not satisfy formula~\eqref{eq:tij}, then $|s_{pqr} - \tilde{x}_i| - |s_{pqr} - \tilde{x}_j| \neq c\Delta t_{ij}$. Hence $I_{ij}(s_{pqr}) = \big| |s_{pqr}-\tilde{x}_i|-|s_{pqr}-\tilde{x}_j|-c\Delta t_{ij} \big| > 0$.
	\end{proof}
	
	Since a single auxiliary indicator function $I_{ij}$ only determines one hyperboloid and cannot uniquely determine the position of the point source, then we define a indicator function
	\[I(s_{pqr}) \coloneqq [I_{21}(s_{pqr}) + I_{31}(s_{pqr}) + I_{41}(s_{pqr}) + I_{51}(s_{pqr})]^{-1}.\]
	
	\begin{theorem}
		Let the source position be denoted by $s^*$. For any $\tilde{y} \in \Omega$, we have
		\[
		I(\tilde{y}) = \begin{cases}
			+\infty, & \tilde{y} = s^*, \\
			\text{a finite positive number}, &\tilde{y} \neq s^*.
		\end{cases}
		\]
	\end{theorem}
	
	\begin{proof}
	 If $y = s^*$, then $I_{21} = I_{31} = I_{41} = I_{51} = 0$, so that the indicator function value $I(\tilde{y}) = +\infty$. If $y \neq s^*$, then at least one of $I_{21}, I_{31}, I_{41}, I_{51}$ is nonzero, so that $I(\tilde{y})$ is a finite positive number.
	\end{proof}
	
	\begin{algorithm}[H]
		\caption{}
		{\bf Step 1:} Set the search domain $\Omega = [x_{\min}, x_{\max}] \times [y_{\min}, y_{\max}] \times [z_{\min}, z_{\max}]$ and then divide $\Omega$ using discrete grid step size.
		
		\noindent{\bf Step 2:} Calculate the theoretical distance difference in advance.
		\[
		d_i = c \cdot \Delta t_{i1}, \quad i = 2, 3, 4, 5.
		\]
		
		\noindent{\bf Step 3:} For each grid point $s_{pqr} = (x_p, y_q, z_r)$:
		\begin{enumerate}
			\item Compute the distance to the observation points $\tilde{x}_i$ for each $i = 1, 2, 3, 4, 5$,
			\[
			r_i(s_{pqr}) = |s_{pqr} - \tilde{x}_i|.
			\]
			\vspace{-25pt}
			\item Compute the distance differences from the reference observation point to the other observation points for each $i = 2, 3, 4, 5$,
			\[
			e_i(s_{pqr}) = | r_i(s_{pqr}) - r_1(s_{pqr}) - d_{i-1} |.
			\]
			
			\item Cumulative total error
			\[
			E(s_{pqr}) = \sum_{i=2}^{5} e_i(s_{pqr}).
			\]
			
			\item Compute the indicator function values
			\[
			I(s_{pqr}) = \frac{1}{E(s_{pqr})}.
			\]
		\end{enumerate}
		
		\noindent{\bf Step 4:} Plotting.
	\end{algorithm}

	\section{Numerical Examples}\label{num}
	
	In this section, we present four numerical examples to validate the proposed algorithm in three dimensional space. In the first example, we show the reconstruction results of a single point source under different noise levels. In the second example, we present the reconstruction results for two point sources under different noise levels. In the third example, we compare the actual trajectory of a moving point source traveling along a spatial sine curve with the reconstructed position of the point source. In the fourth example, our moving trajectory is a spiral curve, and we also compare the actual trajectory with the reconstructed position of the point source. 
	
	All examples are implemented using the same observation point configuration: $\tilde{x}_1 = (0,0,0), \; \tilde{x}_2 = (3,0,0), \; \tilde{x}_3 = (-3,0,0), \; \tilde{x}_4 = (0,3,0), \; \tilde{x}_5 = (0,0,3)$. The wave speed is set to $c = 1$, and the discrete grid step size $h$ is 0.05. The reconstruction error refers to the absolute difference $e = |s - s_r|$ between the reconstructed source position $s_r$ and the true source position $s$ for a single fixed point source. The mean reconstruction error is defined as the mean of the single-point errors over all discrete emission moments, which is given by $e = \frac{1}{N} \sum_{k=1}^{N} |s^{(k)} - s_r^{(k)}|$ where $N$ is the total number of discrete emission moments, and $s^{(k)}$ and $s_r^{(k)}$ denote the true and reconstructed source positions corresponding to the $k$-th emission moment, respectively.

	\textbf{Example 1 (Single fixed point source).} Consider a fixed point source $s = (2,1,-1)$ and domain $\Omega = [-2,3] \times [-2,3] \times [-3,2]$. Figure~\ref{fig1_} shows the reconstruction results of the point source obtained using the indicator function values under noise levels of $0\%, \; 10\%, \; 20\%, \; 30\%$.
	
	\begin{figure}[htbp]
		\hspace*{-0.6cm}
		\centering
		\noindent\includegraphics[width=1.1\textwidth]{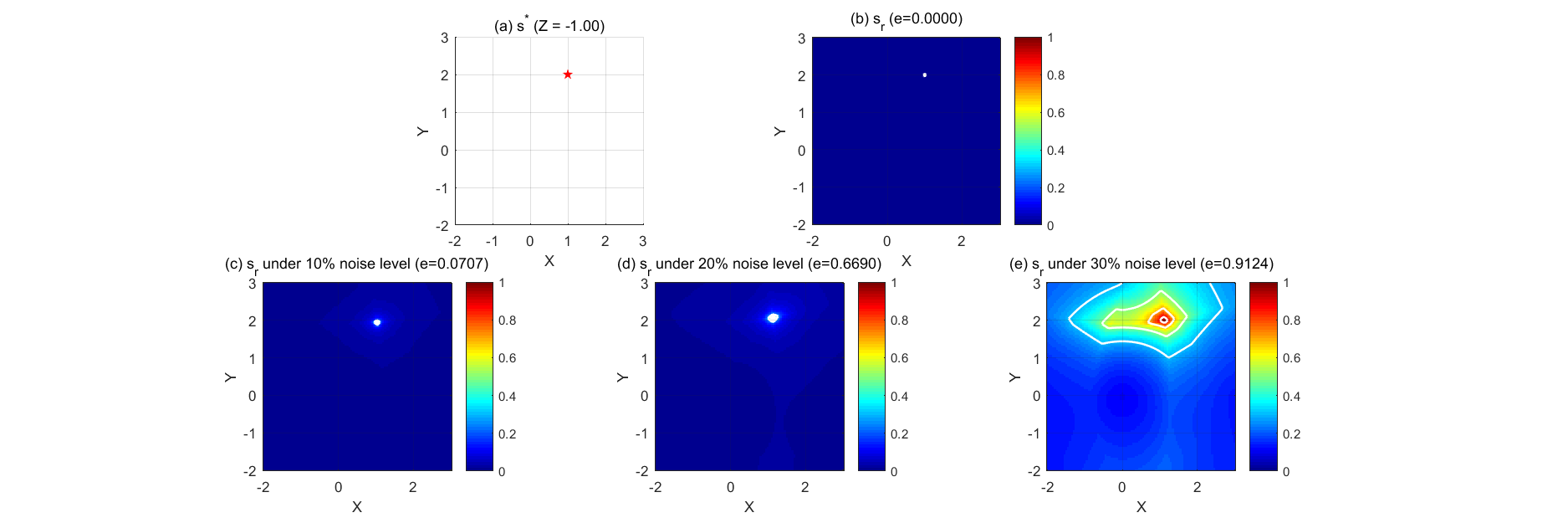}
		\caption{Reconstruction of a fixed point source with five observation points. (a)true source location; (b)(c)(d)(e) reconstruction position and corresponding reconstruction error with different noise level, where (b) $0\%$, (c) $10\%$, (d) $20\%$, (e) $30\%$.}
		\label{fig1_}
	\end{figure}
	
	Figure~\ref{fig1_} presents the reconstruction results for a single point source under different noise levels, providing a visual illustration of how different noise conditions affect the algorithm. In the figure, $e$ represents the reconstruction error. Figure~\ref{fig1_}(a), the true source position is marked by a red five-pointed star, serving as a reference for the inversion results obtained under different noise levels. From the figure, it can be seen that in the absence of noise, the result obtained by the sampling method is very accurate, and the imaging contrast is quite clear. There is only a single isolated bright spot, and the indicator function values in other regions are nearly zero. Since the true point source is selected exactly on a grid node, the reconstruction error is zero in the noise-free case. At a noise level of $10\%$, the peak values of the indicator function remain concentrated without significant shift, and the response values in other regions remain low. The reconstructed result is located near the true source, with an error of only 0.0707. This indicate that the algorithm has good noise resistance. When the noise level increases to $20\%$, 
	we observed that the highlight area expand significantly and is no longer an isolated bright spot. However, the peak does not shift noticeably, and only mild interference appearing in the surrounding area. This suggests that the algorithm's stability is affected by noise, but it is still able to locate the point source fairly accurately. At a noise level of $30\%$, both the highlighted region and red highlighted area expand further, indicating that the source localization becomes blurred and is no longer as accurate as before. The imaging contrast decreases markedly, and the interference region becomes too large, resulting in a final error of 0.9124. This shows that excessive noise levels seriously affect the imaging quality and stability of the algorithm.
	
	\textbf{Example 2 (Two fixed point sources).} Consider two point sources $s_1 = (2,1,3), \; s_2 = (-1,-2,-1)$ at different emission moments with domain $\Omega = [-2,3] \times [-3,2] \times [-2,4]$. Figure~\ref{fig2_} shows the reconstruction results of the point sources using the indicator function values under four noise levels: $0\%, \; 10\%, \; 20\%, and \;30\%$. This visually demonstrates the performance of point source localization as the noise level changes.
	
	\begin{figure}[htbp]
		\hspace*{-0.6cm}
		\centering
		\noindent\includegraphics[width=1.1\textwidth]{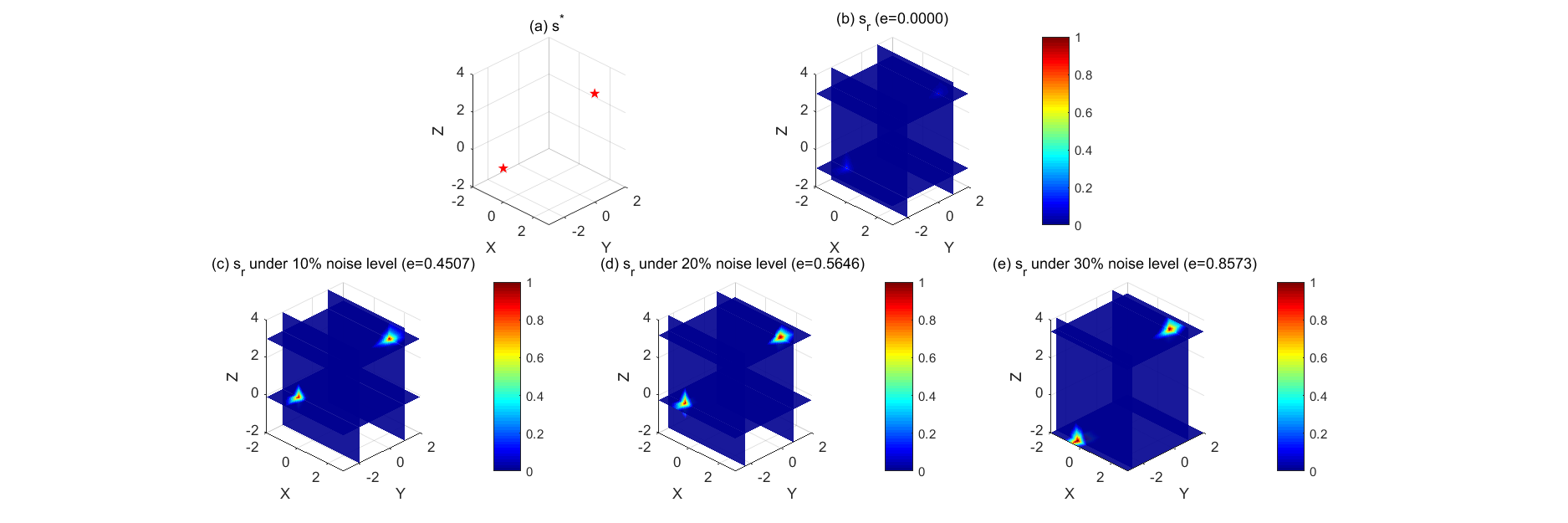}
		\caption{Reconstruction of two fixed point sources with five observation points. (a) true spatial distribution of the two point sources; (b)(c)(d)(e) reconstruction position and corresponding mean reconstruction error with different noise level, where (b) $0\%$, (c) $10\%$, (d) $20\%$, (e) $30\%$.}
		\label{fig2_}
	\end{figure}
	
	Figure~\ref{fig2_} displays the visualization of the indicator function values within the sampling region, where the color intensity reflects the response strength calculated by the localization algorithm. In the figure, $e$ represents the mean reconstruction error. In the noise-free case, the indicator function peaks only at the locations of the two true point sources. The response values in other areas are small, providing a sharp imaging contrast that accurately depicts the spatial positions. The reconstruction results are consistent with the true distribution with zero error. With  $10\%$ noise, the mean error is 0.4507, which is defined as the average distance between the two reconstructed point sources and their corresponding true positions. Although the peak areas of the indicator function expand slightly, no other peaks appear, and the overall peak positions do not shift noticeably. The two point sources can still be accurately inverted, maintaining stable localization performance. When the noise level increases to $20\%$, the mean error rises to 0.5646. The peak areas further expand, and signs of other peaks begin to emerge, leading to a significant decrease in imaging contrast. The errors caused by noise become apparent. The estimated positions can still be obtained, but the imaging quality and localization stability decline significantly. Under strong interference of $30\%$ noise, it can be seen from figure~\ref{fig2_}(e) that the high noise level completely disrupts the data. The effective distance difference constraints fail entirely, severely affecting the results. Consequently, it is no longer possible to reconstruct the positions of the point sources.
	
	Overall, as the noise level increase, the imaging quality of the point sources shows a continuous downward trend. When the noise level is low, the peak of the indicator function is well concentrated near the true source position, the imaging is stable, and the contrast is clear. However, with increasing noise levels, the image contrast gradually decreases, the peak region expands, and the mean error becomes larger. Especially when the noise reaches a high level, the reconstruction results are severely disrupted. This makes it impossible to distinguish the corresponding intensity at each grid point, ultimately leading to the failure of the moving point source localization. This shows that the algorithm can accurately invert the positions of two point sources, but it is highly sensitive to observation noise perturbations. Therefore, the noise level must be controlled within a certain range.
	
	\textbf{Example 3(Spatial sine curve trajectory).} The trajectory is given by $a(t) \coloneqq (3\cos t, 3\sin t, 2.5\sin 2t), \; t \in [0,2 \pi]$, discretized at $t_k = (k - 1) \cdot \frac{2 \pi}{30}$ for $k = 1, 2, \cdots, 30$. The domain is $\Omega = [-4,4] \times [-4,4] \times [-4,4]$. Figure~\ref{fig3_} shows the trajectory reconstruction and the corresponding error when the moving point source moves along this spatial sine curve.
	
	\begin{figure}[htbp]
		\hspace*{-0.6cm}
		\centering
		\noindent\includegraphics[width=1.1\textwidth]{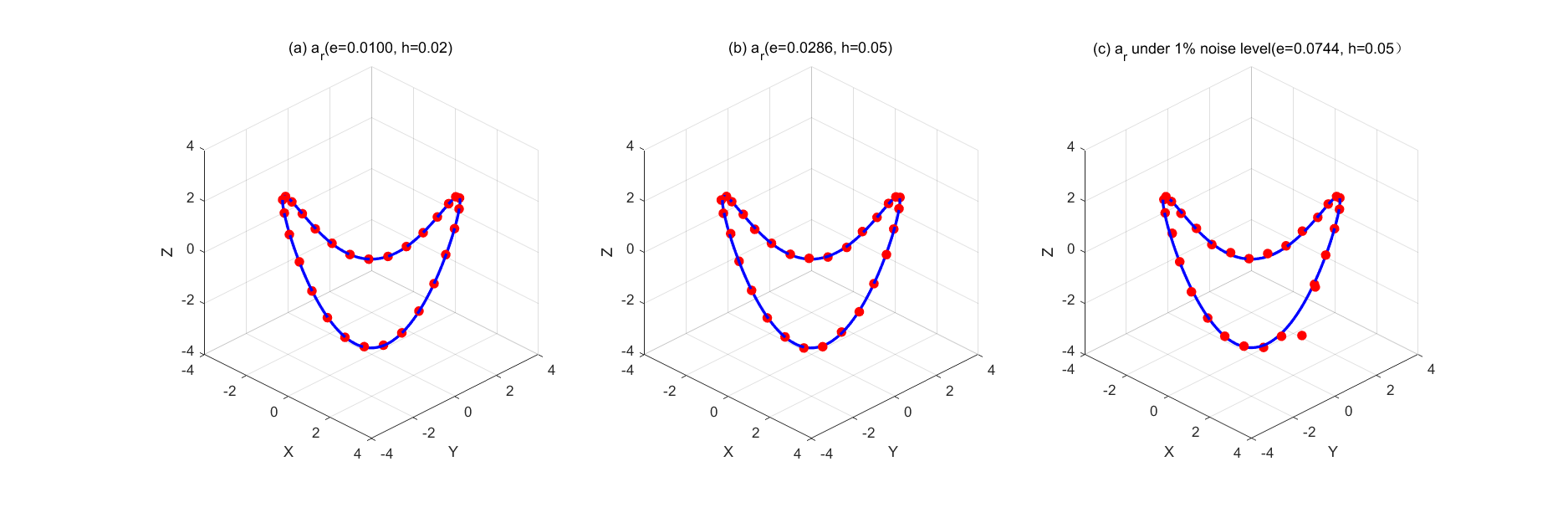}
		\caption{Reconstruction of a spatial sine curve using five observation points. (a) Reconstructed trajectory and mean error under noise-free conditions with a discrete grid step size of $h=0.02$; (b) Reconstructed trajectory and mean error under noise-free conditions with a discrete grid step size of $h=0.05$; (c)Reconstructed trajectory and mean error under $1\%$ noise with a discrete grid step size of $h=0.05$.}
		\label{fig3_}
	\end{figure}
	
	In the Figure~\ref{fig3_}, the blue solid line represents the true trajectory, and the red circles indicate the 30 reconstructed source positions. Both are shown in the same figure to facilitate comparison of the reconstruction accuracy and to observe if there are any areas with excessive errors. In the figure, $e$ is the mean reconstruction error, and $h$ is the discrete grid step size.
	
	From the reconstruction results shown in Figure~\ref{fig3_}(a) and (b), the reconstructed point sources in the noise-free conditions are almost all distributed along the true trajectory. The mean reconstruction error is only 0.0100 and 0.0286. This verifies that the algorithm can accurately reconstruct the trajectory in the absence of noise. The mean error refers to the average distance between each reconstructed point source and its corresponding true point source. The fact that the average error is not zero is mainly due to the discrete grid step size. In the computation, we evaluate the indicator function at the grid nodes, so the reconstructed source position lies on grid nodes. However, the corresponding true source positions may not fall exactly on grid nodes but lie between two grid nodes, which causes error. We observe that although both cases are noise-free, the two reconstruction results have different average errors. This is primarily because the grid size differs between the two cases. A smaller grid step size leads to a smaller average reconstruction error. If higher accuracy is required, the grid step size can be reduced. But it should be noted that this will increase the computational cost.
	
	Figure~\ref{fig3_}(c) shows the reconstruction results after adding $1\%$ relative noise to the measurements, based on a grid step size of 0.05. It can be seen that the red circles are still distributed on or near the true trajectory, with no obvious deviation. The mean reconstruction error under noisy conditions is 0.0744. Even with $1\%$ noise interference, the reconstruction results remain consistent with the true trajectory. This demonstrates that the proposed algorithm is stable with respect to measurement noise.
	
	\textbf{Example 4 (Helical trajectory).} The trajectory is given by $a(t) \coloneqq (2\cos t, 2\sin t, 0.5t), \; t \in [0,5 \pi]$, discretized at $t_k = (k - 1) \cdot \frac{5 \pi}{29}$ for $k = 1, 2, \cdots, 30$. The domain is $\Omega = [-4,4] \times [-4,4] \times [0,8]$. Figure~\ref{fig4_} shows the trajectory reconstruction and the corresponding error when the moving point source moves along this helical curve.
	
	\begin{figure}[htbp]
		\hspace*{-0.6cm}
		\centering
		\noindent\includegraphics[width=1.1\textwidth]{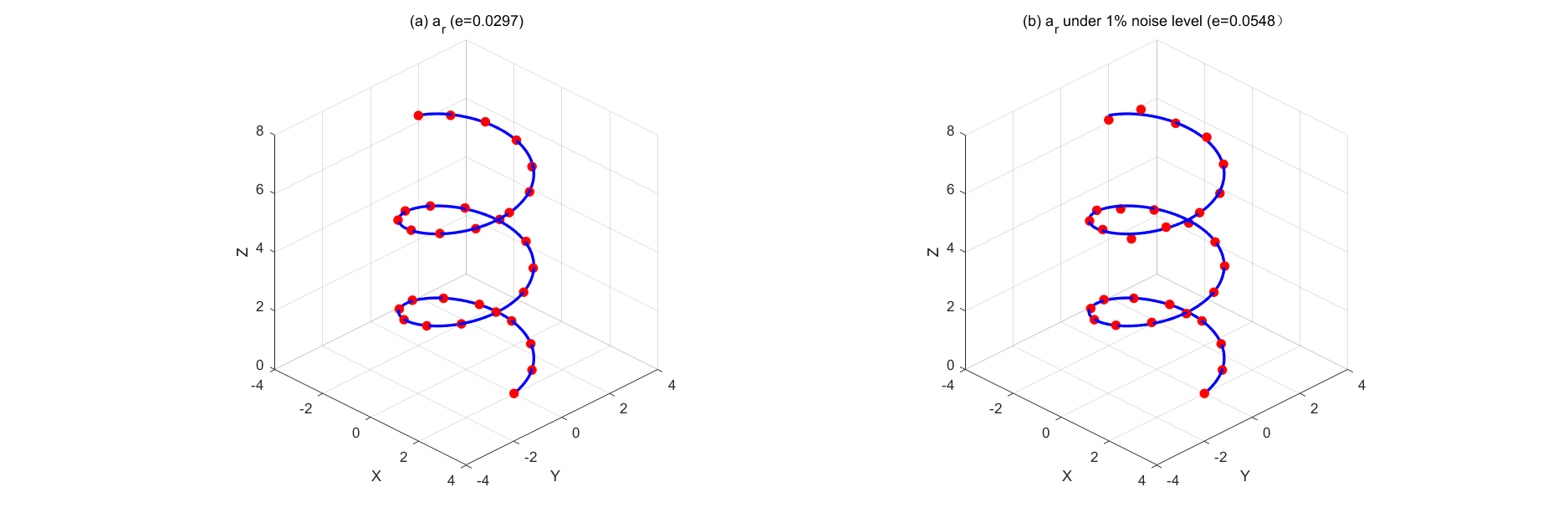}
		\caption{Reconstruction of a helical curve using five observation points. (a) Reconstructed trajectory and mean error without noise; (b) Reconstructed trajectory and mean error with $1\%$ noise.}
		\label{fig4_}
	\end{figure}
	
	In the Figure~\ref{fig4_}, the blue solid line represents the true trajectory, and the red circles indicate the 30 reconstructed source positions.
	
	From the reconstruction results shown in Figure~\ref{fig4_}(a), under noise-free conditions, the reconstructed point sources fall exactly on the true trajectory. There are no locations with large errors. This demonstrates that the proposed algorithm can accurately invert the spatial helical curve. The mean error is only 0.0297, and the error mainly comes from grid step size. Figure~\ref{fig4_}(b) shows the reconstruction results when the noise level increases to $1\%$. It can be seen that a few reconstructed point sources do not lie exactly on the true trajectory, but most of them remain on the trajectory. The average error is 0.0548, which shows that $1\%$ noise interference is relatively small and only affects the reconstruction of a few point sources. This demonstrates that the proposed algorithm has good noise resistance.
	
	\section{Conclusions and future work}\label{con}
	
	In the theoretical part, this paper combines algebraic theory with geometric knowledge. It explains the ambiguity of the solution with four observation points in three-dimensional space and proves the uniqueness of the solution with five observation points. The system of formulas constructed from time differences of arrival is linearized to $Bu = b$. For four observation points, we obtain two symmetric solutions using the rank of the coefficient matrix and geometric constraints. For five points, uniqueness is proved by combining geometric properties and proof by contradiction. The proof need to depend on special arrangements of the observation points.
	
	The proposed algorithm is based on a grid sampling approach. The indicator function \[I(s_{pqr}) \coloneqq [I_{21}(s_{pqr}) + I_{31}(s_{pqr}) + I_{41}(s_{pqr}) + I_{51}(s_{pqr})]^{-1}.\]
	is constructed based on the fact that the residual of the formula $|s^* - \tilde{x}_i| - |s^* - \tilde{x}_1| = c \Delta t_{i1}$ satisfied by the point source $s^*$ should be zero. This reduces the number of observation points and improves the computational efficiency of the algorithm. In the computation, the algorithm does not rely on the invertibility of matrices. Compared with algorithms such as the least squares method, it offers increased stability and is not affected by matrix singularity. Compared with iterative methods, our method does not require a priori knowledge of a relatively accurate initial solution, but only needs a prior exploration region. If the prior region does not contain the true source and localization fails, we can appropriately expand the prior region. 
	
	However, the algorithm does have some limitations. (1) Since waves do not satisfy Huygens' principle when propagating in a two-dimensional plane, the arrival times of a moving point source cannot be effectively measured in 2D. Consequently, the proposed algorithm is no longer applicable. (2) If there are multiple moving point sources in the space, it is difficult to separate the time of arrival corresponding to multiple point sources. Furthermore, this algorithm cannot be applied to supersonic moving point sources, as waves emitted later may catch up with or even overtake waves emitted earlier, disrupting the order of arrival and leading to inversion failure. These limitations will be addressed in our future work.


\begin{thebibliography}{99}
		
		\bibitem{FKK1999} S. Fischer, H. Koorapaty, E. Larsson, and A. Kangas, System performance evaluation of mobile positioning methods, Proc. IEEE Vehicular Technology Conference, Houston, TX, USA, May 1999.
		
		\bibitem{S1999} M. A. Spinto, Further results on GSM mobile station location, IEE Electronics Letter, 35 (1999): no. 22.
		m{CC2003} E. Choi and D. A. Cicci, Analysis of GPS static positioning problems, Applied Mathematics and Computation, 140 (2003): 37-51.
		
		\bibitem{GM2015} F. Grondin and F. Michaud, Time difference of arrival estimation based on binary frequency mask for sound source localization on mobile robots, 2015 IEEE/RSJ International Conference on Intelligent Robots and Systems (IROS), Hamburg, Germany, 2015, pp. 6149-6154.
		
		\bibitem{TOY2021} N. Tsumachi, T. Ohseki, and K. Yamazaki, Base station selection method for RAT-dependent TDOA positioning in mobile network, Proc. IEEE Radio Wireless Symp. (RWS), 2021, pp. 119-122.
		
		\bibitem{CNAS2014} M. Compagnoni, R. Notari, F. Antonacci and A. Sarti, A comprehensive analysis of the geometry of TDOA maps in localization problems, Inverse Problems, 30 (2014): 035004.
		
		\bibitem{GG2003} F. Gustafsson and F. Gunnarsson, Positioning using time-difference of arrival measurements, 2003 IEEE International Conference on Acoustics, Speech, and Signal Processing, 2003. Proceedings. (ICASSP '03), Hong Kong, China, 2003, pp. VI-553.
		
		\bibitem{LQS2012} Y. Liu, T. Qiu and H. Sheng, Time-difference-of-arrival estimation algorithms for cyclostationary signals in impulsive noise, Signal Processing, 92 (2012): 2238-2247.
		
		\bibitem{WSZG2019} P. Wu, S. Su, Z. Zuo, X. Guo, B. Sun and X. Wen, Time Difference of Arrival (TDoA) Localization Combining Weighted Least Squares and Firefly Algorithm, Sensors, 19 (2019): p. 2554.
		
		\bibitem{DZY2020} Z. Deng, H. Wang, X. Zheng, and L. Yin, Base station selection for hybrid TDOA/RTT/DOA positioning in mixed LOS/NLOS environment, Sensors, 20 (2020): p. 4132.
		
		\bibitem{TWC2022} G. Torsoli, M. Z. Win, and A. Conti, Selection of reference base station for TDOA-based localization in 5G and beyond IIoT, Proc. IEEE Globecom Workshops (GC Wkshps), 2022, pp. 317-322.
		
		\bibitem{NK2025} N. K. Inamdar, Time difference of arrival source localization: Exact linear solutions for the general 3D problem, arXiv:2501.01076.	
		
		\bibitem{WHF1976} W. H. FOY, Position-Location Solutions by Taylor-Series Estimation. IEEE Transactions on Aerospace and Electronic Systems, vol. AES-12, no. 2, pp. 187-194, March 1976.
		
		\bibitem{CH1905} Y. T. Chan and K. C. Ho, A simple and efficient estimator for hyperbolic location. IEEE Transactions on Signal Processing, vol. 42, no. 8, pp. 1905-1915, Aug. 1994.
		
		\bibitem{Su2020} Y. Su, X. Fu and N. Zhang, TDOA localization algorithm based on Lagrange constraint factor to modify the initial value of iteration, Advances in Applied Mathematics, 9 (2020): 372-381.
		
		\bibitem{TKI2022} J. Takagi, H. Kanazawa, K. Ichikawa and H. Mitamura, A simple intuitive method for seeking intersections of hyperbolas for acoustic positioning biotelemetry, PLOS ONE, 17 (2022): 1-19.
		
		\bibitem{Li2026} Q. Li, B. Chen, P. Gao, Y. Sun and Y. Sun, Reconstruction of acoustic sources from the initial arrival time of waves, Inverse Problems, 42 (2026): 035006.
		
		
		
		
	\end{thebibliography}
\end{document}